\documentclass[preprint]{elsarticle}

\usepackage[margin=1.5in]{geometry}
\usepackage{amsmath,amssymb,amsthm}
\usepackage{tikz}
\usetikzlibrary{patterns}
\usepackage{comment}
\usepackage[utf8]{inputenc}
\allowdisplaybreaks[4]

\newtheorem{theorem}{Theorem}[section]
\newtheorem{lemma}[theorem]{Lemma}
\newtheorem{coro}[theorem]{Corollary}
\newtheorem{conj}[theorem]{Conjecture}
\newtheorem{prop}[theorem]{Proposition}
\newtheorem{claim}[theorem]{Claim}

\newtheorem{obs}[theorem]{Observation}

\newdefinition{definition}[theorem]{Definition}
\newdefinition{notation}[theorem]{Notation}
\newdefinition
{asp}[theorem]{Assumption}
\newdefinition{example}[theorem]{Example}
\newdefinition{exc}[theorem]{Exercise}

\begin{document}

\title%[monochromatic $k$-connected subgraphs]
{A Ramsey Type problem for highly connected subgraphs}

\author[1]{Chunlok Lo}
\ead{clo42@gatech.edu}
\address[1]{
College of Computing, Georgia Institute of Technology, Atlanta, Georgia, USA 30332}

\author[2]{Hehui Wu\fnref{fn2}}
\ead{hhwu@fudan.edu.cn}
\address[2]{
Shanghai Center for Mathematical Sciences, 
Fudan University, Shanghai, China 200438}
\fntext[fn2]{Supported in part by National Natural Science Foundation of China (Grant No. 11931006), National Key Research and Development Program of China (Grant No. 2020YFA0713200), and the Shanghai Dawn Scholar Program (Grant No. 19SG01).}

\author[3]{Qiqin Xie\fnref{fn3}}
\ead{qqxie@shu.edu.cn}
\address[3]{
Department of Mathematics, College of Science,  
Shanghai University, Shanghai, China 200444}
\fntext[fn3]{Supported in part by National Natural Science Foundation of China (Grant No. 12201390) and National Key R\&D Program of China (Grant No. 2022YFA1006400).}

\begin{abstract}
Bollob\'{a}s and Gy\'{a}rf\'{a}s conjectured that for any $k, n \in \mathbb{Z}^+$ with $n > 4(k-1)$, 
every 2-edge-coloring of the complete graph on $n$ vertices 
leads to a $k$-connected monochromatic subgraph with at least $n-2k+2$ vertices. 
We find a counterexample with $n = \lfloor 5k-2.5-\sqrt{8k-\frac{31}{4}} \rfloor$, 
thus disproving the conjecture, 
and we show the conclusion holds for $n > 5k-2.5-\sqrt{8k-\frac{31}{4}}$ 
when $k \ge 16$. 

\end{abstract}


\maketitle

\section{Introduction}

Ramsey theory is one of the most important research areas in combinatorics. 
For any given integers $s, t$, 
the Ramsey number $R(s, t)$ is the smallest integer $n$, 
such that for any 2-edge-colored (red/blue) $K_n$, 
there must exist a red $K_s$ or a blue $K_t$. 
In 1930, Ramsey \cite{Ramsey30} proved the existence of Ramsey numbers. 
However, estimating Ramsey numbers is known to be notoriously challenging. 

There are many variations of the original Ramsey problem, 
including the one considering highly-connected subgraphs instead of cliques. 
There have been many studies concerning the existence of $k$-connected subgraphs, 
including Mader's \cite{Mader72} result in 1972, 
indicating that every graph with large average degree always contains a $k$-connected subgraph. 

To consider the existence of $k$-connected monochromatic subgraphs in edge-colored complete graphs, 
we let $r_c(k)$ denote the smallest integer such that 
every $c$-edge-colored complete graph on $r_c(k)$ vertices 
must contain a $k$-connected monochromatic subgraph. 
In 1983, Matula \cite{Mat83} proved $2c(k-1)+1 \le r_c(k) < (10/3)c(k-1)+1$. 
Moreover, for 2-edge-coloring, Matula \cite{Mat83} improved the upper bound to $r_2(k) < (3+\sqrt{11/3})(k-1)+1$. 
However, Matula's result does not have any restriction on the order of the $k$-connected monochromatic subgraph. 
In 2008, Bollob\'{a}s and Gy\'{a}rf\'{a}s \cite{BG08} proposed the following conjecture: 

\begin{conj}\label{conjBG}
Let $k, n$ be positive integers. For $n > 4(k-1)$, every 2-edge-colored $K_n$ contains a $k$-connected monochromatic subgraph with at least $n-2k+2$ vertices. 
\end{conj}

Note that the conclusion is not true for $n \le 4(k-1)$ by Matula's result \cite{Mat83} 
(also see \cite{BG08}). 
Moreover, no matter how large $n$ is, 
$n-2k+2$ is the best possible lower bound for the order of the $k$-connected subgraph 
by the example $B(n, k)$ in \cite{BG08}. 
Besides proposing the conjecture, Bollob\'{a}s and Gy\'{a}rf\'{a}s verified the conjecture for $k \le 2$,  and showed it is sufficient to prove the conjecture holds for $4k-3 \le n < 7k-5$. 
Liu, Morris, and Prince \cite{LMP09} verified the conjecture for $k=3$, 
and proved it for $n \ge 13k-15$. 
Later, Fujita and Magnant \cite{FM11} improved the bound to $n > 6.5(k-1)$. 
In 2016, {\L}uczak \cite{Lu15} claimed the proof of the conjecture. 
However, a gap has been found in the proof and not yet fixed \cite{HL} (also see \cite{Ma19}). 

Bollob\'{a}s and Gy\'{a}rf\'{a}s' conjecture could be generalized to multicolored graphs 
(see \cite{LMP08, FM13, KMN17, LW18}). 
Besides, there are some other approaches to force large highly connected subgraphs. 
For example, Fujita, Liu, and Sarkar \cite{FLS16, FLS18} proved the existence of large highly connected subgraphs 
with given independence number. 
The characterization of 2-edge-colored $K_n$ with no large $k$-connected monochromatic subgraphs has also been studied 
(see \cite{JWW14}), 

The main result of this paper is that we show Conjecture~\ref{conjBG} fails for $n = \lfloor 5k-2.5-\sqrt{8k-\frac{31}{4}} \rfloor$. 
On the other hand, we verify the conclusion for any larger $n$.  

\begin{theorem}\label{main}
\quad
\begin{enumerate}
\item Let $k,n \in \mathbb{Z}^+$. 
If $k \ge 16$ and $n > 5k-2.5-\sqrt{8k-\frac{31}{4}}$, 
then for any two spanning subgraphs $G_R$ and $G_B$ of $K_n$, where $E(G_R) \cup E(G_B)$ covers all edges of $K_n$, either $G_R$ or $G_B$ has a $k$-connected subgraph with at least $n-2k+2$ vertices. 
\item For every $k \in \mathbb{Z}^+$, 
let $n = \lfloor 5k-2.5-\sqrt{8k-\frac{31}{4}} \rfloor$. 
There exists a 2-edge-colored $K_n$, 
such that there is no $k$-connected monochromatic subgraph with at least $n-2k+2$ vertices. 
\end{enumerate}
\end{theorem}

Note that in Theorem~\ref{main}~(1), 
$G_R$ and $G_B$ may have common edges. 
%This could be considered as an extension of the general definition of edge-coloring, 
%since in our setting the coloring of each edge is not necessarily unique. 
It is not hard to see that in our statement, it makes no difference to use either the original or our extended definition of edge-coloring, 
which allow every edge to be colored more than once. 

In Section 2, we will give a decomposition structure to graphs with no large $k$-connected monochromatic subgraphs.  
%We will also analyse the structures of graphs with no large $k$-connected monochromatic subgraphs.  
We will prove Theorem \ref{main}~(1) in Section 3 and 4, 
and demonstrate the counterexample (Theorem~\ref{main}~(2)) in Section 5.
%The proof of Lemma \ref{inequ} could be found in Section 5. 

\section{Structures of graphs without large $k$-connected subgraphs}

In this section, 
we first introduce a decomposition for graphs with no $k$-connected subgraphs of large order. %Some ideas of the decomposition were inspired by {\L}uczak's proof \cite{Lu15}. 
We start with some terminologies and notations that we will use throughout this note. 
We follow the notations and terminologies for graphs from \cite{GraphT}.

Let $G = (V, E)$ be a graph. 
$G$ is $k$-connected if and only if it has more than $k$ vertices
and does not have a vertex cut of size at most $k-1$.
For $S \subseteq V$, we use $G[S]$ to denote the subgraph of $G$ induced by $S$. 
We use $N_{G}(S)$ to denote the vertex set $\{v: v\notin S, \exists u \in S, uv \in E(G)\}$, 
and $N_{G}[S]$ to denote $S \cup N_{G}(S)$.
For two sets $S_1$ and $S_2$, we may use $S_1 - S_2$ to denote $S_1 \setminus S_2$. Moreover, we use $G-S$ to denote the subgraph of $G$ induced by $V(G)-S$. 
Let $e = uv$ where $u, v \in V(G)$ and $e \notin E(G)$. We use $G+e$ to denote the graph $(V(G), E(G) \cup \{e\})$.

For $S \subseteq V(G)$, 
we say $S$ is complete (resp. connected) in $G$ if $G[S]$ is a complete (resp. connected) subgraph of $G$. For disjoint $V_1, V_2 \subseteq V(G)$, we say $[V_1, V_2]$ is complete in $G$ if $G$ has a complete bipartite subgraph with partite sets $V_1$, $V_2$. 
We use $E(V_1, V_2)$ to denote the set of edges with one endpoint in $V_1$ and the other endpoint in $V_2$. 

For any positive integer $i$, we use $[i]$ to denote the set of all integers in $[1, i]$. Given a mapping $f$ and a set $X$, we denote $f(X)$ to be $\sum_{x\in X}f(x)$ if $f$ is a real-value function, and denote $f(X)$ to be $\bigcup_{x\in X}f(x)$ if the value of $f$ is a set.

\begin{definition}\label{decomposition}
Let $k \in \mathbb{Z}^+$, $f(k)$ be a non-negative integer. 
Let $G$ be a graph on $n$ vertices, where $n > f(k)+k$. 
We define an {\bf $(f(k), k)$-decomposition} of $G$ to be a sequence of triples $((A_i, C_i, D_i))_{i=1}^{l}$, such that 
\begin{enumerate}
\item $V(G)$ is a disjoint union of $A_1, C_1, D_1$ 
\item $C_i \cup D_i$ is a disjoint union of $A_{i+1}, C_{i+1}, D_{i+1}$, $i \in [l-1]$ 
\item $|C_i| \le k-1$, $i \in [l]$
\item $1 \le |A_i| \le |D_i|$, and there is no edge between $A_i$ and $D_i$, $i \in [l]$  
\item $|C_i|+|D_i| \ge n-f(k)$, $i \in [l-1]$ 
\item $|C_l|+|D_l| < n-f(k)$
\end{enumerate}
\end{definition}

By (1) and (2) of Definition \ref{decomposition}, we have: 

\begin{prop}\label{partition}
$V(G)$ is a disjoint union of $A_1, \dots, A_i, C_i, D_i$ for any $i \in [l]$.
\end{prop}

We also consider edge partitions of $G$ with respect to the decomposition. For convenience, we will frequently use the following notations. 

\begin{notation}\label{edge type}
Let $((A_i, C_i, D_i))_{i=1}^l$ be an $(f(k),k)$-decomposition of $G$. 
\begin{enumerate}
    \item We use $A_{l+1}$ to denote $C_l \cup D_l$. 
    \item We say an edge $uv$ is an $AA$-type edge if there exists $i \in [l+1]$ such that $u, v \in A_i$. 
    We define $AC$-type and $AD$-type for $i \in [l]$ similarly. 
\end{enumerate}
\end{notation}

Thus we have the following propositions. 

\begin{prop}\label{Epartition}
Let $((A_i, C_i, D_i))_{i=1}^l$ be an $(f(k),k)$-decomposition of $G$. 
\begin{enumerate}
    \item $E(G)$ is a disjoint union of $AA$-type and $AC$-type edges. 
    \item All $AD$-type edges are in $\overline{G}$. 
    \item Let $K = G \cup \overline{G}$. Then $E(K)$ is a disjoint union of all $AA$-type, $AC$-type, and $AD$-type edges. 
\end{enumerate}
\end{prop}

\begin{proof}
(2) is followed by Definition \ref{decomposition}(4), and (1) is a corollary of (2) and (3). We only need to prove (3). 

Let $u, v$ be two distinct vertices in $V(G)$. 
By Proposition \ref{partition}, there must exist $i \in [l+1]$, such that $\{u, v\} \cap A_i \ne \emptyset$. 
we take the smallest such $i$. 
By symmetry, we may assume $u \in A_i$. 
Then by proposition \ref{partition}, 
$v$ must be in one of the disjoint sets $A_i$, $C_i$, and $D_i$. 
Thus the type of $uv$ is unique. 
Hence, $E(K)$ is a disjoint union of all $AA$-type, $AC$-type, and $AD$-types edges. 
\end{proof}

\begin{lemma}\label{exist decomposition}
Let $k \in \mathbb{Z}^+$, $f(k)$ be a non-negative function on $k$. 
Let $G$ be a graph on $n$ vertices with $n \ge f(k)+k+1$. 
If $G$ does not have a $k$-connected subgraph with at least $n-f(k)$ vertices, then $G$ has an $(f(k), k)$-decomposition. 
\end{lemma}

\begin{proof}
Let $G_0 = G$. 
Since $f(k)$ is non-negative, $|G_0| = n \ge n-f(k)$. 
We repeat the following steps until $|G_i| < n-f(k)$. 
\begin{enumerate}
\item Let $C_{i+1}$ be a cut of $G_i$ of size at most $k-1$. Since $|G_i| \ge n-f(k) \ge k+1$, there must exist one such cut. 
\item Let $A_{i+1}$ be the vertex set of smallest component of $G_i - C_{i+1}$, and $D_{i+1} = V(G_i) - (A_{i+1} \cup C_{i+1})$. 
\item Let $G_{i+1}$ be the subgraph of $G_i$ induced by $C_{i+1} \cup D_{i+1}$.
\end{enumerate}
The sequence of triples generated by the above procedure is an $(f(k), k)$-decomposition of $G$.
\end{proof}

\begin{definition}\label{strong decomposition}
We say an $(f(k), k)$-decomposition is {\bf strong} if $|A_i|+|C_i| < n-f(k)$, for any $i \in [l]$.
\end{definition}

\begin{lemma}\label{decomposition to no conn}
Let $k \in \mathbb{Z}^+$, $f(k)$ be a non-negative function on $k$. 
Let $G$ be a graph on $n$ vertices, where $n \ge f(k)+k+1$. 
If $G$ has a strong $(f(k), k)$-decomposition, 
then $G$ does not have a $k$-connected subgraph with at least $n-f(k)$ vertices. 
\end{lemma}

\begin{proof}
Let $((A_i, C_i, D_i))_{i=1}^{l}$ be a strong $(f(k), k)$-decomposition of $G$. 
Suppose $G$ has a $k$-connected subgraph $H$ such that $|H| \ge n-f(k)$. 
Let $i^*$ be the smallest $i \in [l]$ such that $A_i \cap V(H) \ne \emptyset$. 
Note that by Proposition \ref{partition} and (6) of Definition \ref{decomposition}, such $i^*$ must exist. 
Then $H$ must be a subgraph of $G(A_{i^*} \cup C_{i^*} \cup D_{i^*})$. 
We claim $V(H) \cap D_{i^*} = \emptyset$. 
Otherwise by (3) and (4) of Definition \ref{decomposition}, $V(H) \cap C_{i^*}$ is a cut of $H$ of size at most $k-1$, 
which is a contradiction to the connectivity of $H$. 
Thus $H$ must be a subgraph of $G(A_{i^*} \cup C_{i^*})$. 
However since the decomposition is strong, $|H| \le |A_{i^*}|+|C_{i^*}| < n-f(k)$. 
We conclude that $G$ does not have a $k$-connected subgraph with at least $n-f(k)$ vertices. 
\end{proof}

For the rest part of this section, we apply the decomposition and its properties on 2-edge-colored complete graphs with no large $k$-connected monochromatic subgraphs. 
This will help us to set up the proof of Theorem \ref{main}(1). 

Let $k, n \in \mathbb{Z}^+$, where $n \ge 4k-3$, 
and $G$ be a complete graph on $n$ vertices. 
We color each edge of $G$ by at least one of color red or blue. 
Note that we allow the edges to be colored red and blue simultaneously. 
Let $R$ (resp. $B$) be the set of all red (resp. blue) edges. 
%Note that $R \cap B$ is not necessarily an empty set. 
We set $G_R = (V, R)$ and $G_B = (V, B)$. 
For $S \subseteq V$, we use $R(S)$ (resp. $B(S)$) to denote the subgraph of $G_R$ (resp. $G_B$) induced by $S$. 
For convenience, We will use $N_{R}(S)$, $N_{B}(S)$, $N_{R}[S]$, and $N_{B}[S]$ to denote $N_{G_R}(S)$, $N_{G_B}(S)$, $N_{G_R}[S]$, and $N_{G_B}[S]$.

Suppose there exists $G$, such that $G$ does not have a $k$-connected monochromatic subgraph with at least $n-2k+2$ vertices. 
By Lemma \ref{exist decomposition}, $G_R$ must have a $(2k-2, k)$-decomposition $((A_i, C_i, D_i))_{i=1}^{l_R}$, 
and $G_B$ must have a $(2k-2, k)$-decomposition $((U_s, X_s, Y_s))_{s=1}^{l_B}$. We choose $G$ and the compositions according to the following rules: 

\begin{asp}\label{decomp asp}
We may assume
\begin{enumerate}
    \item $l_R$ and $l_B$ are maximized; 
    \item With respect to (1), $R$ and $B$ are maximized.  
\end{enumerate}
\end{asp}

Next, we will characterize the decompositions of $G_R$ and $G_B$ with more details. 
Note that the notations $U_{l_B+1}$, $UU$-type, $UX$-type, and $UY$-type are similar to those mentioned in Notation \ref{edge type}. 

\begin{prop}\label{basic prop}
We have the following propositions: 
\begin{enumerate}
    \item All $AD$-type edges are in $\overline{R} \subseteq B$, all $UY$-type edges are in $\overline{B} \subseteq R$, and no edge are both $AD$-type and $UY$-type. 
    \item $A_i$ is connected in $G_{\overline{B}}$ for all $i \in [l_R]$ and $U_s$ is connected in $G_{\overline{R}}$ for all $i \in [l_B]$. 
\end{enumerate} 
\end{prop}

\begin{proof}
(1) is followed by Proposition~\ref{Epartition}. 
For (2), suppose $A_j$ is not connected in $G_{\overline{B}}$ for some $j$ in $[l_R]$. 
Let $A_j^*$ be the vertex set of a connected component of $\overline{B}(A_j)$. 
In other words, $[A_j^*, A_j-A_j^*]$ is complete in $G_B$. 
We obtain $R'$ from $R$ by removing all edges between $A_j^*$ and $A_j-A_j^*$ in $R$. 
For all $i < j$, we set $(A_i', C_i', D_i') = (A_i, C_i, D_i)$. 
We set $(A_j', C_j', D_j') = (A_j^*, C_j, D_j \cup A_j - A_j^*)$ and $(A_{j+1}', C_{j+1}', D_{j+1}') = (A_j - A_j^*, C_j, D_j)$. 
And for all $i > j$, we set $(A_{i+1}', C_{i+1}', D_{i+1}') = (A_i, C_i, D_i)$. 
Then $((A_i', C_i', D_i'))_{i' = 1}^{l_R+1}$ is a decomposition of $G_{R'}$, 
which is a contradiction to Assumption \ref{decomp asp}(1). 
Thus, $A_i$ is connected in $G_{\overline{B}}$ for all $i \in [l_R]$. 
Similarly, we can prove $U_s$ is connected in $G_{\overline{R}}$ for all $s \in [l_B]$. 
\end{proof}

By Proposition \ref{basic prop}(1), we have the following corollary:

\begin{coro}\label{AD UY empty}
For any $i \in [l_R]$ and $s \in [l_B]$, 
\begin{enumerate}
    \item Either $A_i \cap U_s$ or $D_i \cap Y_s$ is empty;
    \item Either $A_i \cap Y_s$ or $U_s \cap D_i$ is empty. 
\end{enumerate}
\end{coro}

\begin{proof}
For (1), suppose $A_i \cap U_s$ and $D_i \cap Y_s$ are both non-empty for some $i \in [l_R]$ and $s \in [l_B]$. 
Let $u \in A_i \cap Y_s$ and $v \in D_i \cap U_s$. 
Then $uv$ is an $AD$-type and $UY$-type edge simultaneously,
which is a contradiction to Proposition \ref{basic prop}(1). 
Similarly we can prove (2). 
\end{proof}

\begin{prop}\label{B ext}
Suppose exists $i \in [l_R]$ such that $B(C_i \cup D_i)$ has a $k$-connected subgraph $H$ of order at least $2k-1$, 
then $B(A_{i} \cup V(H))$ is $k$-connected. 
\end{prop}

\begin{proof}
Since $H$ is a subgraph of $B(C_i \cup D_i)$ and $|C_i| \le k-1$, 
we have $|V(H) \cap D_i| = |V(H)| - |V(H) \cap C_i| \ge (2k-1) - (k-1) = k$. 
By Proposition \ref{basic prop} (1), $[A_i, V(H) \cap D_i]$ is complete in $B(A_{i} \cup V(H))$. 
Thus, $B(A_{i} \cup V(H))$ is $k$-connected. 
\end{proof}

By Definition \ref{decomposition} (2), $C_i \cup D_i = A_{i+1} \cup C_{i+1} \cup D_{i+1}$, $i \in [l_R-1]$. 
We will have the following corollary if we apply Proposition \ref{B ext} recursively: 

\begin{coro}\label{B ext induc}
Suppose there exists $i \in [l_R]$ such that $B(C_i \cup D_i)$ has a $k$-connected subgraph $H$ of order at least $2k-1$, 
then $B(A_1 \cup A_2 \cup \dots \cup A_{i} \cup V(H))$ is $k$-connected. 
\end{coro}

\begin{claim}\label{Ai Us bound}
$|A_i| \le k-1$, $\forall i \in [l_R]$. 
$|U_s| \le k-1$, $\forall s \in [l_B]$. 
\end{claim}

\begin{proof}
Suppose there exists $i$ such that $|A_i| \ge k$.
By (4) of Definition \ref{decomposition}, $|D_i| \ge |A_i| \ge k$. 
Since $[A_i, D_i]$ is complete in $G_B$, we have $B(A_i \cup D_i)$ is $k$-connected. 
If $i = 1$, $B(A_1 \cup D_1)$ is a $k$-connected subgraph of $B$. 
If $i \ge 2$, 
By (2) of Definition \ref{decomposition}, $B(A_i \cup D_i)$ is a $k$-connected subgraph of $B(C_{i-1} \cup D_{i-1})$. 
Moreover, $|B(A_i \cup D_i)| = |A_i| + |D_i| \ge k + k > 2k-1$. 
Thus by applying Corollary \ref{B ext induc} on $(i-1)$ and  $H = B(A_i \cup D_i)$, 
we have $B(A_1 \cup A_2 \cup \dots \cup A_{i} \cup D_{i})$ is $k$-connected. 
However, by Proposition \ref{partition} and Definition \ref{decomposition}(3), 
$|A_1 \cup A_2 \cup \dots \cup A_{i} \cup D_{i}| = |V(G)| - |C_i| \ge n - (k-1) \ge n - 2k + 2$, a contradiction. 
Thus, $|A_i| \le k-1$, $\forall i \in [l_R]$. 
By symmetry, $|U_s| \le k-1$, $\forall s \in [l_B]$. 
\end{proof}

For convenience, we use $A$ (resp. $U$) to denote $A_1 \cup A_2 \cup \dots \cup A_{l_R}$ (resp. $U_1 \cup U_2 \cup \dots \cup U_{l_B}$). 
Combining (5)(6) of Definition \ref{decomposition} and Proposition \ref{partition}, we have the following corollary: 

\begin{coro}\label{A U bound}
$2k-1 \le |A| \le 3k-3$. 
$2k-1 \le |U| \le 3k-3$.
\end{coro}

\begin{proof}
By Definition \ref{decomposition} (6) and Proposition \ref{partition}, 
$\sum_{i=1}^{l_R} |A_i| = n - (|C_{l_R}|+|D_{l_R}|) \ge n - (n-2k+1) = 2k-1$. 
By Definition \ref{decomposition} (5) and Proposition \ref{partition}, 
$\sum_{i=1}^{l_R} |A_i| = (\sum_{i=1}^{l_R-1} |A_i|)+|A_{l_R}| = n - (|C_{l_R-1}|+|D_{l_R-1}|) + |A_{l_R}| \le n - (n-2k+2) + (k-1) = 3k-3$. 
By symmetry, $2k-1 \le |U| \le 3k-3$. 
\end{proof}

The last two claims of this section are leaded by the maximality of $R$ and $B$. 

\begin{claim}\label{RB}
$R$ is the disjoint union of all $AA$-type and $AC$-type edges, and $B$ is the disjoint union of all $UU$-type and $UX$-type edges. 
\end{claim}

\begin{proof}
We will first prove that all $AA$-type edges are in $R$. 
Suppose there exists $i' \in [l_R+1]$ and $u, v \in A_{i'}$, such that $uv \notin R$.
Consider $R' = R + uv$. 
Since there does not exist $i^* \in [l_{R}]$ such that $uv$ is an edge between $A_{i^*}$ and $D_{i^*}$, 
$((A_i, C_i, D_i))_{i=1}^{l_R}$ is still a $(2k-2, k)$-decomposition of $G_{R'}$. 
Moreover, by Claim \ref{Ai Us bound} and (3) of Definition \ref{decomposition}, 
for any $i \in [l_R]$, $|A_i|+|C_i| \le (k-1) + (k-1) \le n-2k+2$ since $n \ge 4k-3$. 
Hence, the decomposition is strong.
And by Lemma \ref{decomposition to no conn}, 
$G_{R'}$ does not have a $k$-connected subgraph with at least $n-2k+2$ vertices, a contradiction to the maximality of $R$. 
Thus all $AA$-type edges are in $R$. 

Similarly, we can prove all $AC$-type edges are in $R$, 
and all $UU$-type and $UX$-type edges are in $B$. 
By Proposition \ref{Epartition}, 
$R$ is the disjoint union of all $AA$-type and $AC$-type edges, and $B$ is the disjoint union of all $UU$-type and $UX$-type edges. 
\end{proof}

\begin{claim}\label{cut size}
$|C_i| = |X_s| = k-1$ for all $i \in [l_R]$ and $s \in [l_B]$. \end{claim}

\begin{proof}
Suppose there exists $i' \in [l_R]$, such that $|C_{i'}| < k-1$. 
Let $u$ be a vertex in $D_{i'}$. 
Consider $R'$ to be the edge set which consists $R$ and all edges between $A_{i'}$ and $u$. 
Let $C'_{i'} = C_{i'} \cup \{u\}$, and $C'_i = C_i$ for all $i \ne i'$. 
Then $((A_i, C'_i, D_i))_{i=1}^{l_R}$ is a strong $(2k-2, k)$-decomposition of $G_{R'}$. 
By Lemma \ref{decomposition to no conn}, 
$G_{R'}$ does not have a $k$-connected subgraph with at least $n-2k+2$ vertices, a contradiction to the maximality of $R$. 
Thus, $|C_i| = k-1$ for all $i \in [l_R]$. 
By symmetry, $|X_s| = k-1$ for all $s \in [l_B]$. 
\end{proof}

\section{Proof of Theorem~\ref{main}(1)}

In this section, 
we prove Theorem~\ref{main}(1).  Suppose it is not true, there exists a 2-edge-colored $K_n$
that has no $k$-connected monochromatic subgraph with at least $n-2k+2$ vertices, where $n$ and $k$ are integers satisfy $n> 5k-2.5-\sqrt{8k-\frac{31}{4}}$ and $k\ge 16$. Note that by the examples we mentioned in Section 1 (see \cite{BG08, Mat83}), 
we may assume $n \ge 4k-3$. 
We follow all the assumptions and claims in section 2 in the proof of Theorem~\ref{main}(1).

We start with the following Observation:
\begin{obs}
\begin{equation*}\label{EQ:alledges}(k-1)(|A|+|U|)+\sum_{i=1}^{l_R+1} \binom{|A_i|}{2}+\sum_{s=1}^{l_B+1}\binom{|U_s|}{2}=|R|+|B|=\binom{n}{2}+|R\cap B|.
\end{equation*}
\end{obs}
\begin{proof}
By claim~\ref{RB} and claim~\ref{cut size}, we have $|R|=\sum_{i=1}^{l_R} |A_i||C_i|+\sum_{i=1}^{l_R+1}\binom{|A_i|}{2}=(k-1)|A|+\sum_{i=1}^{l_R+1}\binom{|A_i|}{2}$ and $|B|=\sum_{s=1}^{l_B} |U_s||X_s|+\sum_{s=1}^{l_B+1}\binom{|U_s|}{2}=(k-1)|U|+\sum_{s=1}^{l_B+1}\binom{|U_s|}{2}$. Sum them up, we have the above equation.
\end{proof}

\begin{definition}
In $R\cap B$, let $P$ consist of all edges that are both $AC$-type and $UX$-type, all the $AC$-type edges in $E(U_{l_B+1},U_{l_B+1})$, and all the $UX$-type edges in $E(A_{l_R+1}, A_{l_R+1})$. Note that $U_{l_B+1}=V-U$ and $A_{l_R+1}=V-A$. 

Let $i \in [l_R]$ and $s \in [l_B]$. Given a vertex $v\in A_i\cap U_s$, let $Q_R(v)$ be the family of edges $uv$ with $u$ in $A_i\cap Y_s$. Similarly, we define $Q_B(v)$ and let $Q(v)=Q_R(v)\cup Q_B(v)$. 
Let $Q_R = \cup_{v\in A\cap U}Q_R(v)$, $Q_B = \cup_{v\in A\cap U}Q_B(v)$, and $Q=Q_R\cup Q_B$. 
\end{definition}

The following is the key formula for our reminding argument in this paper.
\begin{claim}\label{CLM:Formula}
\begin{align*}
    &(5k-3-n)|A\cap U|-(2k-1)-\frac{1}{2}|A\cap U|^2
    \\=&(|A|-2k+1)(|U|-2k+1)
+(|A|+|U|-4k+2)(k-|A\cap U|)\\
&+\sum_{i=1}^{l_R}\sum_{s=1}^{l_B}((k-1)|A_i\cap U_s|-\frac{1}{2}|A_i\cap U_s|^2-|Q(A_i\cap U_s)|)+|P|. 
\end{align*}
\end{claim}
\begin{proof}
Note that in $R-B$, $Q_R$ consists of all $UY$-type edges in $E[A_i,A_i]$ for $i\in [l_R]$, and in $B-R$, $Q_B$ consists of all $AD$-type edges of $E[U_s,U_s]$ for $s\in [l_B]$. Moreover, by the definition of $P$,  $R\cap B-P$ consists of the edges in $E(A_i, A_i)$ for $i\in [l_R]$ and edges in $E(U_s, U_s)$ for $s\in [l_B]$ but not in $Q$, and the edges in $E[V-A-U, V-A-U]$. Therefore, we have 
$$|R\cap B|-|P|=\binom{|V-A-U|}{2}+\sum_{i=1}^{l_R} \binom{|A_i|}{2}+\sum_{s=1}^{l_B}\binom{|U_s|}{2}-\sum_{i=1}^{l_R}\sum_{s=1}^{l_B}\binom{|A_i\cap U_s|}{2}-|Q|,$$ 
Sum up with Observation~\ref{EQ:alledges}, we can get
\begin{align*}
&(k-1)(|A|+|U|)+\binom{|V-A|}{2}+\binom{|V-U|}{2}-|P|\\
=&\binom{n}{2}+\binom{|V-A-U|}{2}-\sum_{i=1}^{l_R}\sum_{s=1}^{l_B}\binom{|A_i\cap U_s|}{2}-|Q|.
\end{align*}

Therefore
\begin{align*}
&\sum_{i=1}^{l_R}\sum_{s=1}^{l_B}\binom{|A_i\cap U_s|}{2}+|Q|-|P|\\
= &\binom{n}{2}+\binom{n-|A\cup U|}{2}-\binom{n-|A|}{2}-\binom{n-|U|}{2}-(k-1)(|A|+|U|)\\
=&\frac{2n-1}{2}(-|A\cup U|+|A|+|U|)+\frac{1}{2}(|A\cup U|^2-|A|^2-|U|^2)
-(k-1)(|A|+|U|)\\
=&(n-\frac{1}{2})|A\cap U|+\frac{1}{2}((|A|+|U|-|A\cap U|)^2-|A|^2-|U|^2)-(k-1)(|A|+|U|)\\
=&(n-\frac{1}{2})|A\cap U|+|A||U|-(|A|+|U|)|A\cap U|+\frac{1}{2}|A\cap U|^2-(k-1)(|A|+|U|)\\
=&(2k-1)+|A \cap U|(n-4k+1.5) + \frac{1}{2}|A\cap U|^2+(|A|-2k+1)(|U|-2k+1)\\
&+(|A|+|U|-2(2k-1))(k-|A\cap U|).
\end{align*}
The last equation can be verified by expansion, and now we can use the property that $A$ and $U$ have size at least $2k-1$ by Corollary~\ref{A U bound}. 

As 
\begin{align*}\sum_{i=1}^{l_R}\sum_{s=1}^{l_B}\binom{|A_i\cap U_s|}{2}=&
\sum_{i=1}^{l_R}\sum_{s=1}^{l_B}(\frac{1}{2}|A_i\cap U_s|^2-(k-1)|A_i\cap U_s|+(k-1.5)|A_i\cap U_s|)\\
=&(k-1.5)|A\cap U|+\sum_{i=1}^{l_R}\sum_{s=1}^{l_B}(\frac{1}{2}|A_i\cap U_s|^2-(k-1)|A_i\cap U_s|)
\end{align*}
and 
$$Q = \sum_{i=1}^{l_R}\sum_{s=1}^{l_B}Q(A_i \cap U_s),$$
we have 
\begin{align*}
    &(5k-3-n)|A\cap U|-(2k-1)-\frac{1}{2}|A\cap U|^2
    \\=&(|A|-2k+1)(|U|-2k+1)
+(|A|+|U|-4k+2)(k-|A\cap U|)\\
&+\sum_{i=1}^{l_R}\sum_{s=1}^{l_B}((k-1)|A_i\cap U_s|-\frac{1}{2}|A_i\cap U_s|^2-|Q(A_i\cap U_s)|)+|P|. 
\end{align*}
\end{proof}

For $i\in [l_R]$ and $s\in [l_B]$ with $A_i\cap U_s\not=\emptyset$,  if $D_i\cap U_s=\emptyset$, then we put $A_i\cap U_s$ in a set $A^*_i$, if $A_i\cap Y_s=\emptyset$, then we put $A_i\cap U_s$ in a set $U^*_s$. If both of $A_i\cap Y_s$ and $D_i\cap U_s$ are empty, then we arbitrarily put $A_i\cap U_s$ in $A^*_i$ or $U^*_s$. Let $A^*=\bigcup_{i=1}^{l_R} A^*_i$ an let $U^*=\bigcup_{s=1}^{l_B} U^*_s$. 

By the fact that every edge is either in $R$ or $B$, we immediately have the following claim.
\begin{claim}\label{CLM:A*U*}
The followings are true for $A^*$ and $U^*$:
\begin{enumerate}
\item $A\cap U=A^*\sqcup U^*$.
\item $Q(v)=Q_R(v)$ for all $v\in A^*$ and $Q(v)=Q_B(v)$ for all $v\in U^*$.
\end{enumerate}
\end{claim}

\begin{proof}
(1) By Corollary \ref{AD UY empty}(2), $A_i\cap U_s\subseteq A^*_i$ or $A_i\cap U_s\subseteq U^*_s$. Also, by the definition of $A^*_i$ and $U^*_s$, we have $A^*_i\cap U^*_s$ is empty. Hence, $A\cap U=A^*\sqcup U^*$. 

(2) If $v\in A^*_i$, suppose $v\in A_i\cap U_s$, by definition of $A^*_i$, we have $D_i\cap U_s=\emptyset$, therefore $Q_B(v)=\emptyset$, hence $Q(v)=Q_R(v)$. Similarly, if $v\in U^*$, we have $Q(v)=Q_B(v)$.
\end{proof}

\begin{claim}
For any $i\in [l_R], s\in [l_B]$, we have to following:
\begin{enumerate}
\item $|P|\ge \sum_{i}|A_i-U||C_i-U|+\sum_{s}|U_s-A||X_s-A|$.
\item $\sum_{A_i\cap U_s\subseteq A^*_i}((k-1)|A_i\cap U_s|-\frac{1}{2}|A_i\cap U_s|^2-|Q(A_i\cap U_s)|)
\ge |A^*_i|(k-1-|A_i|)+\frac{1}{2}|A^*_i|^2.$
\item $\sum_{A_i\cap U_s\subseteq U^*_s}((k-1)|A_i\cap U_s|-\frac{1}{2}|A_i\cap U_s|^2-|Q(A_i\cap U_s)|)
\ge |U^*_s|(k-1-|U_s|)+\frac{1}{2}|U^*_s|^2.$
\end{enumerate}
\end{claim}
\begin{proof}
(1) follows from the definition of $P$, that it contains all the $AC$-type edges in $E(U_{l_B+1},U_{l_B+1})$, and the $UX$-type edges in $E(A_{l_R+1}, A_{l_R+1})$. 

By symmetric, we only need to show (2).
\begin{align*}
    &\sum_{A_i\cap U_s\subseteq A^*_i}(|A_i\cap U_s||A_i|-|Q(A_i\cap U_s)|)\\
    =&\sum_{A_i\cap U_s\subseteq A^*_i}|A_i\cap U_s|(|A_i|-|A_i\cap Y_s|)\\
    \ge&\sum_{A_i\cap U_s\subseteq A^*_i}|A_i\cap U_s|\sum_{\substack{A_i\cap U_t\subseteq A^*_i\\t\le s}}|A_i\cap U_t|\\
    =&\left(\sum_{A_i\cap U_s\subseteq A^*_i}|A_i\cap U_s|\right)^2+\frac{1}{2}\sum_{A_i\cap U_s\subseteq A^*_i}|A_i\cap U_s|^2\\
    =&\frac{1}{2}|A^*_i|^2+\frac{1}{2}\sum_{A_i\cap U_s\subseteq A^*_i}|A_i\cap U_s|^2
\end{align*}
Therefore, we have 
\begin{align*}
&\sum_{A_i\cap U_s\subseteq A^*_i}((k-1)|A_i\cap U_s|-\frac{1}{2}|A_i\cap U_s|^2-|Q(A_i\cap U_s)|)\\
&\ge \sum_{A_i\cap U_s\subseteq A^*_i}(|A_i\cap U_s|(k-1-|A_i|)+(|A_i\cap U_s||A_i|-|Q(A_i\cap U_s)|-\frac{1}{2}|A_i\cap U_s|^2))\\
&\ge |A^*_i|(k-1-|A_i|)+\frac{1}{2}|A^*_i|^2.
\end{align*}
\end{proof}

\begin{claim}
    $\sum_{i=1}^{l_R} |A_i^*|(k-1-|A_i|)+\frac{1}{2}|A_i^*|^2+\sum_{s=1}^{l_B} |U_s^*|(k-1-|U_s|)+\frac{1}{2}|U_s^*|^2\ge \frac{1}{12}|A\cap U|^2$
\end{claim}
\begin{proof}
Assume $|A_{j_1}^*|\ge |A_{j_2}^*|\ge |A_{j_3}^*|\dots$, we have 
    \begin{align*}
         &\sum_{i=1}^{l_R} |A_i^*|(k-1-|A_i|)+\frac{1}{2}\sum_{i=1}^{l_R}|A_i^*|^2\\
         \ge &|A_{j_1}^*|(k-1-(k-1))+|A_{j_2}^*|(k-1-(k-1))+|A_{j_3}^*|(k-1-(|A|-2(k-1)))\\
         &+\sum_{i=4}^{l_R}|A_{j_i}^*|(k-1-0)+\frac{1}{2}\sum_{i=1}^{l_R}|A_i^*|^2\\
        \ge & \sum_{i=4}^{l_R}|A_{j_i}^*|(k-1)+\frac{1}{2}\sum_{i=1}^{l_R}|A_i^*|^2\\
        \ge &\frac{1}{6}(|A_{j_1}^*|+|A_{j_2}^*|+|A_{j_3}^*|)^2+\frac{1}{3}\sum_{i=4}^{l_R}|A_{j_i}^*|(\sum_{i=1}^{l_R}|A_{j_i}^*|)\\
        \ge & \frac{1}{6}(\sum_{i=1}^{l_R}|A_{j_i}^*|)^2=\frac{1}{6}|A^*|^2.
    \end{align*}
    Similarly, we have $\sum_{s=1}^{l_B} |U_s^*|(k-1-|U_s|)+\frac{1}{2}|U_s^*|^2\ge \frac{1}{6}|U^*|^2$. Hence $$\sum_{i=1}^{l_R} |A_i^*|(k-1-|A_i|)+\frac{1}{2}|A_i^*|^2+\sum_{s=1}^{l_B} |U_s^*|(k-1-|U_s|)+\frac{1}{2}|U_s^*|^2\ge \frac{1}{6}(|A^*|^2+|U^*|^2)\ge \frac{1}{12}|A\cap U|^2.$$
\end{proof}

Together with Claim~\ref{CLM:Formula}, we immediately have the following:
\begin{coro}
\begin{align*}
    &(5k-3-n)|A\cap U|-(2k-1)-\frac{7}{12}|A\cap U|^2\\
    \ge& (|A|-2k+1)(|U|-2k+1)
+(|A|+|U|-4k+2)(k-|A\cap U|).
\end{align*}
\end{coro}
For convenience, let $\lambda=5k-3-n$. Since $n > 5k-2.5-\sqrt{8k-\frac{31}{4}}$, we have $\lambda<\sqrt{8k-\frac{31}{4}}-\frac{1}{2}$, hence $\lambda^2+\lambda<8k-8$ if $\lambda\ge 0$.

\begin{claim}
$|A \cap U|\le k-1$.
\end{claim}

\begin{proof}
If $|A\cap U|\ge k$, as $0\le |A|-2k+1, |U|-2k+1\le k-1$,  we have 
\begin{align*}
    &\lambda|A\cap U|-(2k-1)-\frac{7}{12}|A\cap U|^2\\
    \ge& (|A|-2k+1)(|U|-2k+1)+(|A|+|U|-4k+2)(k-|A\cap U|)\\
=&((|A|-2k+1)+(k-|A\cap U|))((|U|-2k+1)+(k-|A\cap U|))- (k-|A\cap U|)^2\\
\ge& ((k-1)+k-|A\cap U|)(k-|A\cap U|)- (k-|A\cap U|)^2\\
=&(k-1)(k-|A\cap U|).
\end{align*}
That is, 
$$0\ge (k^2+k-1)-(k-1+\lambda)|A\cap U|+\frac{7}{12}|A\cap U|^2.$$
So we should have $k-1+\lambda>0$ and $(k-1+\lambda)^2\ge \frac{7}{3}(k^2+k-1)$, hence $\lambda\ge \lceil\sqrt{\frac{7}{3}(k^2+k-1)}-(k-1)\rceil\ge \sqrt{8k-\frac{31}{4}}-\frac{1}{2}$. The last inequality is true when $k\ge 16$.

Thus we must have $|A\cap U|\le k-1$.
\begin{comment}In this case, we have
$$\lambda |A \cap U|\ge (2k-1)  + \frac{7}{12}|A \cap U|^2.$$

As $\lambda<\sqrt{8k-\frac{31}{4}}-\frac{1}{2}<2\sqrt{2k-1}$, we will have 
$$0\ge \lambda^2-4\lambda|A \cap U| + \frac{7}{3}|A \cap U|^2.$$
Therefore $|A \cap U|< (1+\frac{\sqrt{2}}{2})\lambda$.
\end{comment}
\end{proof}

Let $\tau=k-1-|V-A-U|$, let $I^*=\{i: A^*_i\not=\emptyset\}$, let $S^*=\{s: U^*_s\not=\emptyset\}$.

\begin{claim}
We have the following:
\begin{enumerate}
\item There exists a $i_1\in I^*$, such that  $$\sum_{i=1}^{l_R}(|A^*_i|(k-1-|A_i|)+\frac{1}{2}|A^*_i|^2+|A_i-U||C_i-U|)
\ge \frac{1}{2}|A^*|^2+\sum_{i\not=i_1}|A^*_i|(n-|A|-|U|).$$
\item There exists a $s_1\in S^*$, such that  $$\sum_{s=1}^{l_B}(|U^*_s|(k-1-|U_s|)+\frac{1}{2}|U^*_s|^2+|U_s-A||X_s-A|)
\ge \frac{1}{2}|U^*|^2+\sum_{s\not=s_1}|U^*_s|(n-|A|-|U|).$$
\end{enumerate}
\end{claim}

\begin{proof}
By symmetry, we just need to prove (1).

For $i \in [l_R]$ with $A_i\cap U\not=\emptyset$, in particular, suppose $A_i\cap U_s\not=\emptyset$ with $s \in [l_B]$. By Corollary \ref{AD UY empty}(1),  we have $D_i\cap Y_s=\emptyset$, therefore $D_i-U\subseteq X_s$, hence $|D_i-U|\le |X_s|\le k-1$. We have $|C_i-U|=\sum_{j=i+1}^{l_R+1}|A_j-U|-|D_i-U|\ge \sum_{j=i+1}^{l_R}|A_j-U|+|V-A-U|-(k-1)\ge \sum_{j=i+1}^{l_R}|A_j-U|-\tau$.

Let $i_0=\max\{i\in I^*: \sum_{j=i}^{l_R}|A_j-U|>\tau\}$, and let $a=\sum_{j=i_0}^{l_R}|A_j-U|-\tau$. Let $I^*_{<i_0}=\{i\in I^*, i<i_0\}$.

We have $\sum_{i=1}^{l_R}|A_i-U||C_i-U|\ge \sum_{i\in I^*_{<i_0}}|A_i-U|(\sum_{i<j<i_0, j\in I^*}|A_j-U|+a)=\sum_{i, j \in I^*_{<i_0}, i<j}|A_i-U||A_j-U|+\sum_{i\in I^*_{<i_0}}|A_i-U|a$.
So

\begin{align*}
&\sum_{i=1}^{l_R}(|A^*_i|(k-1-|A_i|)+\frac{1}{2}|A^*_i|^2+|A_i-U||C_i-U|)\\
\ge &\sum_{i=1}^{l_R}(|A^*_i|(k-1-|A_i\cap U|))+\frac{1}{2}\sum_{i=1}^{l^R}|A_i^*|^2-\sum_{i=1}^{l^R}|A^*_i||A_i-U|\\
&+\sum_{i,j\in I^*_{<i_0}, i<j}|A_i-U||A_j-U|+\sum_{i\in I^*_{<i_0}}|A_i-U|a\\
\ge &\sum_{i=1}^{l_R}(|A^*_i|(k-1-|A_i\cap U|))+\frac{1}{2}\sum_{i=1}^{l_R}|A_i^*|^2-\sum_{i=i_0}^{l_R}|A^*_i||A_i-U|\\
& +\sum_{i \in I^*_{<i_0}}(a-|A^*_i|)|A_i-U|+\sum_{i,j\in I^*_{<i_0}, i<j}|A_i-U||A_j-U|
\end{align*}

Let $f=\sum_{I^*_{<i_0}}(a-|A^*_i|)|A_i-U|+\sum_{i,j\in I^*_{<i_0}, i<j}|A_i-U||A_j-U|$. We consider $f$ as a function of $\{|A_i-U|:i\in I^*_{<i_0}\}$, where $|A_i-U|$ is range from 0 to $k-1-|A_i\cap U|$. Note that $f$ is linear for each $|A_i-U|$ with $1\le i<i_0$, it should achieve its extremal value at its end points. In particular, when $f$ achieve its minimal value, we should have $|A_i-U|=0$ if $\frac{\partial f}{\partial |A_i-U|}>0$ and $|A_i-U|=k-1-|A_i\cap U|$ if $\frac{\partial f}{\partial |A_i-U|}<0$. Suppose $|A_{i_1}-U|=k-1-|A_{i_1}\cap U|$, then for $i\not=i_1$, since $\frac{\partial f}{\partial |A_i-U|}=a-|A_i^*|+\sum_{j\not=i, j<i_0}|A_j-U|\ge (k-1-|A_{i_1}\cap U|)-|A_i^*|\ge k-1-|A\cap U|>0$, we should have $|A_i-U|=0$ for all $1\le i\le i_0$ with $i\not=i_1$. Therefore we have $f\ge (a-|A^*_{i_1}|)(k-1-|A_{i_1}\cap U|)$. If there is no such an $i_1$, we have $|A_i-U|=0$ for all $1\le i<i_0$, hence $f\ge 0$.  

Case 1. There exists such an $i_1$. We have 
\begin{align*}
&\sum_{i=1}^{l_R}(|A^*_i|(k-1-|A_i|)+\frac{1}{2}|A^*_i|^2+|A_i-U||C_i-U|)\\
\ge &\sum_{i=1}^{l_R}|A^*_i|(k-1-|A_i\cap U|)+\frac{1}{2}\sum_{i=1}^{l_R}|A_i^*|^2-\sum_{i=i_0}^{l_R}|A^*_i||A_i-U|+f\\
\ge &\sum_{i=1}^{l_R}|A^*_i|(k-1-|A_i\cap U|)+\frac{1}{2}\sum_{i=1}^{l_R}|A_i^*|^2-\sum_{i=i_0}^{l_R}|A^*_i||A_i-U|+(a-|A^*_{i_1}|)(k-1-|A_{i_1}\cap U|)\\
\ge &\sum_{i\not=i_1}|A^*_i|(k-1-|A_i\cap U|)+\frac{1}{2}\sum_{i=1}^{l_R}|A_i^*|^2-\sum_{i=i_0}^{l_R}|A^*_i|(a+\tau)+a(k-1-|A_{i_1}\cap U|)\\
\ge &\sum_{i\not=i_1}|A^*_i|(|V-A-U|-|A\cap U|+\sum_{j\not= i}|A_j\cap U|))+\frac{1}{2}\sum_{i=1}^{l_R}|A_i^*|^2+a(k-1-|A\cap U|)\\
\ge &\sum_{i\not=i_1}|A^*_i|(|V-A-U|-|A\cap U|+\sum_{j\not= i}|A_j\cap U|)+\frac{1}{2}\sum_{i=1}^{l_R}|A_i^*|^2\\
\ge &\sum_{i\not=i_1}|A^*_i|(n-|A|-|U|)+\sum_{1\le i<j\le l_R}|A_i^*||A_j^*|+\frac{1}{2}\sum_{i=1}^{l_R}|A_i^*|^2\\
= & \frac{1}{2}|A^*|^2+\sum_{i\not=i_1}|A^*_i|(n-|A|-|U|).
\end{align*}

Case 2. Suppose there is no such an $i_1$, thus $f\ge 0$. Let $i_1=i_0$. We have 
\begin{align*}
&\sum_{i=1}^{l_R}(|A^*_i|(k-1-|A_i|)+\frac{1}{2}|A^*_i|^2+|A_i-U||C_i-U|)\\
\ge &\sum_{i=1}^{l_R}|A^*_i|(k-1-|A_i\cap U|)+\frac{1}{2}\sum_{i=1}^{l_R}|A_i^*|^2-\sum_{i=i_0}^{l_R}|A^*_i||A_i-U|+f\\
\ge &\sum_{i=1}^{l_R}|A^*_i|(k-1-|A_i\cap U|)+\frac{1}{2}\sum_{i=1}^{l_R}|A_i^*|^2-\sum_{i\not=i_0}|A^*_i|\tau-|A_{i_0}^*||A_{i_0}-U|\\
\ge &\sum_{i\not=i_0}|A^*_i|(k-1-\tau-|A_i\cap U|)+\frac{1}{2}\sum_{i=1}^{l_R}|A_i^*|^2+|A^*_{i_0}|(k-1-|A_{i_0}|)\\
\ge &\sum_{i\not=i_0}|A^*_i|(|V-A-U|-|A\cap U|+\sum_{j\not= i}|A_j\cap U|)+\frac{1}{2}\sum_{i=1}^{l_R}|A_i^*|^2\\
\ge &\sum_{i\not=i_0}|A^*_i|(n-|A|-|U|)+\sum_{1\le i<j\le l_R}|A_i^*||A_j^*|+\frac{1}{2}\sum_{i=1}^{l_R}|A_i^*|^2\\
= & \frac{1}{2}|A^*|^2+\sum_{i\not=i_1}|A^*_i|(n-|A|-|U|).
\end{align*}
\end{proof}

\begin{claim}
$|I^*|=1, |S^*|=1, |A|=|U|=2k-1$, and $|A\cap U|\le \lambda$.
\end{claim}
\begin{proof}
We have 
\begin{align*}
&\lambda|A\cap U|\\
\ge & (2k-1)+ \frac{1}{2}|A\cap U|^2
+(|A|+|U|-4k+2))(k-|A\cap U|)\\
&+\frac{1}{2}|A^*|^2+\frac{1}{2}|U^*|^2+\sum_{i\not=i_1}|A^*_i|(n-|A|-|U|)+\sum_{s\not=s_1}|U^*_s|(n-|A|-|U|)\\
\ge& (2k-1)+ \frac{1}{2}|A\cap U|^2+\frac{1}{4}|A\cap U|^2\\
&+(|A|+|U|-4k+2))(k-|A\cap U|)+(n-|A|-|U|)(\sum_{i\not=i_1}|A^*_i|+\sum_{s\not=s_1}|U^*_s|).
\end{align*}
If $|A|+|U|-4k+2\ge 1$, we have 
$$\lambda|A\cap U|\ge (2k-1)+ \frac{3}{4}|A\cap U|^2+k-|A\cap U|.$$
Therefore $(\lambda+1)^2-3(3k-1)\ge 0$, we will have $\lambda\ge \sqrt{9k-3}-1> \sqrt{8k-\frac{31}{4}}-\frac{1}{2}$ for all positive integer $k$. So we have $|A|=|U|=2k-1$.

If $|I^*|\ge 2$ or $|S^*|\ge 2$, we will have 
$$\lambda|A\cap U|\ge (2k-1)+ \frac{3}{4}|A\cap U|^2+(5k-3-\lambda-(4k-2))=3k-2-\lambda+\frac{3}{4}|A\cap U|^2.$$
We will have $\lambda^2-3(3k-2-\lambda)\ge 0.$ Hence $\lambda\ge \sqrt{9k-\frac{15}{4}}+\frac{3}{2}> \sqrt{8k-\frac{31}{4}}-\frac{1}{2}$ for all positive integer $k$. 

Furthermore, if one of $I^*$ and $U^*$ is empty, we will have
$\lambda|A\cap U|
\ge  (2k-1)+ \frac{1}{2}|A\cap U|^2
+\frac{1}{2}(|A^*|^2+|U^*|^2)=2k-1+|A\cap U|^2$.
 We will have $\lambda^2-4(2k-1)\ge 0.$ Hence $\lambda\ge \sqrt{8k-4}> \sqrt{8k-\frac{31}{4}}-\frac{1}{2}$ for all positive integer $k$. 
Thus we should have $|I^*|=|S^*|=1$.

Furthermore, we have
$$\lambda |A \cap U|\ge (2k-1)  + \frac{3}{4}|A \cap U|^2.$$

As $\lambda<\sqrt{8k-\frac{31}{4}}-\frac{1}{2}<2\sqrt{2k-1}$, we will have 
$$0\ge \lambda^2-4\lambda|A \cap U| + 3|A \cap U|^2=(|A\cap U|-\lambda)(3|A\cap U|-\lambda).$$
Therefore $|A \cap U|<\lambda$.
\end{proof}\begin{claim}\label{I=S=1}
If $|I^*|=|S^*|=1$, then we must have $n \le 5k-2.5-\sqrt{8k-\frac{31}{4}}$. 
\end{claim}

We will leave the proof of Claim \ref{I=S=1} to the next section, 
which will complete the proof of Theorem \ref{main} (1).

\section{The case when $|I^*|=|S^*|=1$}
Assume $I^*=\{j\}$ and $S^*=\{t\}$. We have $A^*=A_j^*, U^*=U_t^*$ and $A\cap U=A_j^*\sqcup U_j^*$. We have 
\begin{align*}
0\ge & (2k-1)-\lambda(|A^*|+|U^*|)+|A^*|^2+|A^*||U^*|+|U^*|^2+ |A^*|(k-1-|A_j|)\\
&+|U^*|(k-1-|U_t|)+|A_j-U||C_j-A|+|U_t-A||X_t-U|.
\end{align*}

We say $x>_R y$ if $x\in A_i$ and $y\in A_{i'}$ and $i>i'$. A set $X>_R Y$ if $\forall x\in X, y\in Y$ we have $x>_R y$. Similarly we define $\ge_R$, $>_B$, $\ge_B$.

As we assume $\lambda < \sqrt{8k-\frac{31}{4}}-\frac{1}{2}$, we have $\lambda^2+\lambda\le 8k-4$, hence $2k-1>\frac{\lambda^2}{4}$. Note that by Corollary \ref{AD UY empty}(2), either $U_t\cap D_j=\emptyset$ or $A_j\cap Y_t=\emptyset$.

\begin{claim}\label{CLM:UtinCj}
We have the following proposition:
\begin{enumerate}
    \item If $U_t\cap D_j=\emptyset$ then $U_t\subseteq C_j$. If $A_j\cap Y_t=\emptyset$, then $A_j\subseteq X_t$. In particular, we always have $A_j\cap U_t=\emptyset$, $A^*=A_j\cap U$, $U^*=U_t\cap A$. 
    \item $C_j\subseteq U, X_t\subseteq A$.
    \item $V-A-U\subseteq D_j\cap Y_t$. 
\end{enumerate}
\end{claim}
\begin{proof}
(1) By symmetric, We just need to consider the case that $U_t\cap D_j=\emptyset$. By the definition of $A_j^*$ and $U_t^*$, we may assume $A_j\cap U_t$ is a subset of $A_j^*$ if it is not empty, hence $A^*=A_j\cap U$. If there is a vertex $v\in A_i\cap U_t$ with $i<j$, then $R=D_i\cap ((A_j-U)\cup (V-A-U))$ has size at least $|A_j-U|+|V-A-U|-|C_j|$. The edges between $v$ and $R$ are $AD$-type and therefore should be in $B$, and hence $R\subseteq X_t-U$. 
So $|X_t-U|\ge |A_j-U|+|V-A-U|-|C_j|\ge |A_j|-(|A \cap U|-|U^*|)+n-|A\cup U|-(k-1)=|A_j|+|U^*|+(4k-2-\lambda)-|A|-|U|=|A_j|+|U^*|-\lambda$.

We have 
\begin{align*}
0\ge & (2k-1)-\lambda(|A^*|+|U^*|)+|A^*|^2+|A^*||U^*|+|U^*|^2+ |A^*|(k-1-|A_j|)\\
&+|U^*|(k-1-|U_t|)+|U_t-A||X_t-U|.
\end{align*}
If $|A_j|\le \lambda-|U^*|$, then we will have 
\begin{align*}
0\ge & (2k-1)-\lambda(|A^*|+|U^*|)+|A^*|^2+|A^*||U^*|+|U^*|^2+ |A^*|(k-1-\lambda+|U^*|)\\
> & \frac{\lambda^2}{4}-\lambda(|A^*|+|U^*|)+|A^*|^2+2|A^*||U^*|+|U^*|^2\\
= & (|A^*|+|U^*|-\frac{\lambda}{2})^2>0
\end{align*}
Contradiction! We have 
\begin{align*}
0\ge & (2k-1)-\lambda(|A^*|+|U^*|)+|A^*|^2+|A^*||U^*|+|U^*|^2+ |A^*|(k-1-|A_j|)\\
&+|U^*|(k-1-|U_t|)+|U_t-A|(|A_j|+|U^*|-\lambda).
\end{align*}
The right hand side above can be consider as a linear function of $|A_j|$, with range from $\lambda-|U^*|$ to $k-1$, hence it is minimality should achieved at its endpoint. So we just need to consider the case $|A_j|=k-1$. We have
\begin{align*}
0\ge & (2k-1)-\lambda(|A^*|+|U^*|)+|A^*|^2+|A^*||U^*|+|U^*|^2+|U^*|(k-1-|U_t|)\\
+&|U_t-U_t\cap A_j-U^*|(k-1+|U^*|-\lambda).
\end{align*}
If $|U_t|\le |A^*|+|U^*|$, then 
\begin{align*}
0\ge & (2k-1)-\lambda(|A^*|+|U^*|)+|A^*|^2+|A^*||U^*|+|U^*|^2+|U^*|(k-1-|A^*|-|U^*|)\\
> & \frac{\lambda^2}{4}-\lambda|A^*|+|A^*|^2+(k-1-\lambda)|U^*|\\
\ge& (|A^*|-\frac{\lambda}{2})^2\ge 0
\end{align*}
Contradiction! We have
\begin{align*}
0\ge & (2k-1)-\lambda(|A^*|+|U^*|)+|A^*|^2+|A^*||U^*|+|U^*|^2+|U^*|(k-1-|U_t|)\\
+&(|U_t|-|A^*|-|U^*|)(k-1+|U^*|-\lambda).
\end{align*}
Again, The right hand side above can be consider as a linear function of $|U_t|$, with range from $|A^*|+|U^*|$ to $k-1$, hence it is minimality should achieved at its endpoint. So we just need to consider the case $|U_t|=k-1$. We have
\begin{align*}
0\ge & (2k-1)-\lambda(|A^*|+|U^*|)+|A^*|^2+|A^*||U^*|+|U^*|^2\\
&+(k-1-|A^*|-|U^*|)(k-1-\lambda+|U^*|)\\
\ge & (2k-1)+(k-1)(k-1-\lambda)-(k-1-2\lambda)|A^*|+|A^*|^2
\end{align*}
So we should have $(k-1-2\lambda)^2\ge 4(2k-1)+4(k-1)(k-1-\lambda)$, so
$$0\ge 3(k-1)^2+4(2k-1)-4\lambda^2\ge 3(k-1)^2+3(2k-1)>0.$$ Contradiction.

From above, we should have  every vertex $v \in U_t$, $v\ge_R A_j$.

If there exists $v\in A_j\cap U_t\not=\emptyset$, but there is no $AD$-type edge of $v$ in $U_t$, which contradict to Proposition~\ref{basic prop}(2) that $U_t$ is connected in $G_{\overline{R}}$. So $U_t>_R A_j$, therefore $U_t\subset C_j$. Hence $A_j\cap U_t=\emptyset$. 

(2) As $A_j\cap U_t=\emptyset$, and $A\cap U=A^*_j\sqcup U^*_t$, we have $A^*_j=A_j\cap U$ and $U^*_t=U_t\cap A$.

If $|C_j-U|\ge 1$, we will have 
$|A^*|(k-1-|A_j|)+|A_j-U||C_j-U|\ge |A^*|(k-1-|A_j|)+|A_j|-|A^*|\ge k-1-|A^*|$.
Hence we have 
\begin{align*}
0\ge &2k-1-\lambda(|A^*|+|U^*|)+|A^*|^2+|U^*|^2 +|A^*||U^*|+k-1-|A^*|\\
=&\frac{1}{4}\lambda^2-(\frac{1}{2}|A^*|+|U^*|)\lambda+(\frac{1}{2}|A^*|+|U^*|)^2+\frac{3}{4}|A^*|^2-(1+\frac{1}{2}\lambda)|A^*|+3k-2-\frac{1}{4}\lambda^2\\
=&(\frac{1}{2}\lambda-(\frac{1}{2}|A^*|+|U^*|))^2+\frac{3}{4}(|A^*|-\frac{2}{3}-\frac{1}{3}\lambda)^2+3k-2-\frac{1}{4}\lambda^2-\frac{3}{4}(\frac{2}{3}+\frac{1}{3}\lambda)^2\\
\ge& 3k-2-\frac{1}{3}(\lambda^2+\lambda+1)\\
\ge& 3k-2-\frac{1}{3}(8k-7)\\
=&\frac{k}{3}+\frac{1}{3}
>0
\end{align*}
Contradiction! So we have $C_j\subseteq U$. Similarly, we have $X_t\subseteq A$. 

(3) is implied by (2) immediately. 
\end{proof}

Recall that $n=5k-3-\lambda$, $|V-A-U|=k-1-\tau$. Let $\omega=|A_j\cap U|=|A^*|$, $\sigma=|U_t\cap A|=|U^*|$, then $|A\cap U|=\omega+\sigma$. $n=|A|+|U|+|V- A- U|-|A\cap U|=2k-1+2k-1+k-1-\tau-(\omega+\sigma)=5k-3-(\tau+\omega+\sigma)$. So $\lambda=5k-3-n=\tau+\omega+\sigma$. 

Let $A'=A-A_j-(U_t\cap A)$, $U'=U-U_t-(A_j\cap U)$.
\begin{claim}
$|A_j|\ge k-\sigma$, $|U_t|\ge k-\omega$, $|A'|\le k-1$, $|U'|\le k-1$.   
\end{claim}
\begin{proof}
\begin{align*}
0>&\frac{\lambda^2}{4}-\lambda(\omega+\sigma)+\omega^2+\omega\sigma+\sigma^2+\omega(k-1-|A_j|)+\sigma(k-1-|U_t|)\\
=& (\frac{1}{2}\lambda-(\omega+\sigma))^2-\omega\sigma+\omega(k-1-|A_j|)+\sigma(k-1-|U_t|)\\
\ge&\omega(k-1-|A_j|)+\sigma(k-1-|U_t|)-\omega\sigma.
\end{align*}

We have $k-1-|A_j|<\sigma$, hence $|A_j|\ge k-\sigma$. $|A'|=|A|-|A_j|-|U_t\cap A|\le 2k-1-(k-\sigma)-\sigma=k-1$. Similarly we have $|U_t|\ge k-\omega$, $|U'|\le k-1$.
\end{proof}

By the following observation, we may assume $\tau(\omega+\sigma)+\frac{\omega+\sigma}{2}< 2k-2$.
\begin{obs}     
     If $\tau(\omega+\sigma)+\frac{\omega+\sigma}{2}\ge 2k-2$, then $\lambda=\tau+\omega+\sigma\ge \sqrt{8k-\frac{31}{4}}-\frac{1}{2}$.
\end{obs}
\begin{proof}
Since $\tau, \omega, \sigma$ are all integers, we have 
 $$(\tau+\frac{1}{2}+(\omega+\sigma))^2=4(\tau(\omega+\sigma)+\frac{\omega+\sigma}{2})+(\tau+\frac{1}{2}-(\omega+\sigma))^2\ge 4(2k-2)+\frac{1}{4}=8k-\frac{31}{4}.$$ Therefore $\tau+\omega+\sigma\ge \sqrt{8k-\frac{31}{4}}-\frac{1}{2}$.
\end{proof}

Given a vertex $v\in A_i$ with $1\le i\le l_R$, let $R(v)$ be the edges between $v$ and $C_i$. For a set $S\subseteq A$, let $R(S)=\cup_{v\in S}R(v)$. Similarly we define $B(v)$ and $B(S)$ for $v\in U, S\subseteq U$ in the blue graph.

\begin{center}
\resizebox{8.5cm}{!}
{
\begin{tikzpicture}

\draw (-5,2.5) ellipse (1.5 and 1);
%\draw (-5,2.5) ellipse (1.5 and 1);
\node at (-5,2.6) {$V-A-U$};
%\node at (-5,2.2) {$k-1-\tau$};

\filldraw  [fill = gray!50](-8.15,0.5) ellipse (0.75 and 0.5);
\node at (-8.15,0.5) {$U_t \cap A$};
\draw  (-7.6,0.35) ellipse (1.5 and 1);
\node at (-6.2,-0.45) {$C_j$};
\draw  plot[smooth] coordinates {(-7.3,1.35) (-7.3,-0.65)};
\node at (-7.7,-0.35) {$U_t$};
\node at (-6.7,0.2) {$C_j \cap U'$};

%\node at (-8.4,-0.05) {$\sigma$};
%\node at (-7.05,0.05) {$k-1-\sigma$};

\draw (-2.45,0.35) ellipse (1.5 and 1);
\node at (-0.6,-0.45) {$U'-C_j$};
%\node at (-2.45,0.05) {$k-\omega$};

\draw  (-6.7,-2.5) ellipse (1.5 and 1);
\node at (-5.5,-1.55) {$A'$};
%\node at (-6.7,-2.8) {$k-1$};

\fill  [fill = gray!50](-3.3,-2.5) ellipse (1.5 and 1);
\fill  [fill = white] (-3,-3.6) rectangle (-4.9,-1.4);
\draw  (-3.3,-2.5) ellipse (1.5 and 1);
\node at (-2.1,-1.55) {$A_{j}$};
\draw plot[smooth] coordinates {(-3,-1.5) (-3,-3.5)};
%\node at (-3.8,-2.8) {$k-\sigma-\omega$};
\node at (-2.4,-2.5) {$A_j \cap U$};
%\node at (-2.4,-2.9) {$\omega$};
%\fill (-2.3, -2) circle(0.05);
%\node at (-2.6,-2) {$z$};

%\filldraw [pattern={dots},pattern color=gray] (1,0.2) rectangle (2,0.8);
%\node at (2.8, 0.5) {$C_j$};
\filldraw [fill = gray!50] (0.5,2.2) rectangle (1.5,2.8);
\node at (2.5, 2.5) {$A \cap U$};

\end{tikzpicture}
}
\end{center}

W.L.O.G., we assume $U_t\cap D_j=\emptyset$. By Claim~\ref{CLM:UtinCj}, we have $U_t\subseteq C_j\subseteq U$. We have the following claim.
\begin{claim}\label{CLM:DBCount}
$|R(A)|+|B(A_j\cap U)|+|B(U'-C_j)|\ge |E(A',A_j\cap U)|+|E(A_j,C_j)|+|E(A,U'-C_j)|+|E(A_j\cap U,V-A-U)|+|E(U_t\cap A, V-A-U)|$.
\end{claim}
\begin{proof}
We have $U'-C_j=U-U_t-(A_j\cap U)-C_j=U-C_j-(A_j\cap U)$. Hence $U=(A_j\cap U)\sqcup (U'-C_j)\sqcup C_j$. The left hand side count all the $AC$-type edges and $UX$-type edges except $B(C_j)$. All edges of $E(A',A_j\cap U)$, $E(A_j,C_j)$, $E(A, U'-C_j)$, $E(A_j\cap U,V-A-U)$ and $E(U_t\cap A, V-A-U)$ are disjoint and either $AC$-type or $UX$-type (as they are not $AA$-type or $UU$-type, with only possible exceptions in $E(A_j, C_j)$, which are $AC$-type), but does not contain edges in $B(C_j)$. Here the only non-trivial case are for edges between $U_t\cap A$ and $U'-C_j$ in $E(A, U'-C_j)$ and edges in $E(U_t\cap A, V-A-U)$, for which case we use the fact that $X_t\subseteq A$. Hence we have the above claim.
\end{proof}

\begin{claim}
We have the following:
\begin{enumerate}
    \item $\sigma\tau+\omega\tau+\omega\ge 2k-1.$
    \item $|A'|=k-1$.
\end{enumerate}

\end{claim}
\begin{proof}
By definition, $E(A_j,C_j)=R(A_j)$, $A=A'\sqcup A_j\sqcup (U_t\cap A)$, so we have 
$|R(A)|-|E(A_j,C_j)|=|R(A')|+|R(U_t\cap A)|$.

$|R(U_t\cap A)|-|E(U_t\cap A,V-A-U)|=\sigma(k-1-|V-A-U|)=\sigma\tau$. $|B(A_j\cap U)|-|E(A_j\cap U,V-A-U)|=\omega(k-1-|V-A-U|)=\omega\tau$.

Since $U_t\subset C_j\subset U$ and $U=U_t\sqcup (A_j\cap U)\sqcup U'$, we have $|U'-C_j|=|U|-|A_j\cap U|-|C_j|=2k-1-\omega-(k-1)=k-\omega$.

By the claim~\ref{CLM:DBCount}, we have $(|R(A)|-|E(A_j,C_j)|-|E(U_t\cap A, V-A-U)|)+(|B(A_j\cap U)|-|E(A_j\cap U,V-A-U)|)
\ge |E(A',A_j\cap U)|+(|E(A,U'-C_j)|-B(U'-C_j)|)$.

That is,
$$|R(A')|+\sigma\tau+\omega\tau\ge |A'|\omega+(2k-1-(k-1))|U'-C_j|.$$
Therefore
$$|A'|(k-1-\omega)+\sigma\tau+\omega\tau\ge (k-\omega)k.$$
As $|A'|\le k-1$, we have $$(k-1)(k-1-\omega)+\sigma\tau+\omega\tau\ge k(k-\omega),$$ 
that is
$$\sigma\tau+\omega\tau+\omega\ge 2k-1.$$
If $|A'|\le k-2$, we will have 
$$(k-2)(k-1-\omega)+\sigma\tau+\omega\tau\ge k(k-\omega),$$ that is $\sigma\tau+\omega\tau+2\omega\ge 3k-2$, hence $$\sigma+\omega+\tau+2\ge 2\sqrt{(\sigma+\omega)(\tau+2)}\ge 2\sqrt{3k-2}.$$ 
However $$\sigma+\omega+\tau+2=\lambda+2<\sqrt{8k-\frac{31}{4}}+1.5\le 2\sqrt{3k-2}.$$ Contradiction. So we should have $|A'|=k-1$.

\end{proof}

\begin{claim}
$\sigma\tau+\tau+\omega\ge k$.
\end{claim}
\begin{proof}
Among vertices in $A_j\cap U$, let $z\in U_s$ be the one with minimum $s$. Since $C_j\subset U$, we have $V-A-U\subseteq D_j$, hence by Proposition~\ref{basic prop}, $V-A-U\subset X_s$. For every vertex $u$ in $(U'-C_j)\cap U_{>s}$, since $uz$ is not in $R$, we have $u \in X_s$. 
So there are at most $k-1-|V-A-U|=\tau$ vertices of $(U'-C_j)\cap U_{>s}$. Therefore there are at least $|U'-C_j|-\tau=k-\omega-\tau$ vertices in $(U'-C_j)\cap U_{\le s}$. 

Since $V-A-U\subseteq Y_t$, $E(A\cap U_t, V-A-U)\subseteq R(A\cap U_t)$, there are at most $\sigma\tau$ in $R(A\cap U_t)-E(A\cap U_t, V-A-U)$. 

If $k-\omega-\tau>\sigma\tau$, there is at least one vertex $u$ in $(U'-C_j)\cap U_{\le s}$ with no edge in $R(A\cap U_t)$. Note that $u\in U'-C_j\subseteq V-A$, all the edges between $u$ and $A\cap U_t$ should be in B. Note that $X_t\subseteq A$, we have $u\not\in X_t$. 
Suppose $u\in U_{s'}$ for some $s'\le s$, we have $A\cap U_t\subseteq X_{s'}$. As $u\in U'-C_j$, the edges between $u$ and $A_j$ are all in $B$ and $A_j\ge_B u$. If $s'<s$, we also have $A_j\subseteq X_{s'}$. But $|A_j\cup (A\cap U_t)|=2k-1-|A'|=k>|X_{s'}|$.
If $s'=s$, then $(V-A-U)\cup (A_j-U_s)\cup (A\cap U_t)\subseteq X_s$, we have $|X_s|\ge |A_j|+|A\cap U_t|+|V-A-U|-|A_j\cap U_s|\ge k+k-\tau-\omega> k$. Contradiction!
So we must have $\sigma\tau+\tau+\omega\ge k$.
\end{proof}

Now we are ready to complete the proof that $n=5k-3-(\omega+\sigma+\tau)\le 5k-2.5-\sqrt{8k-\frac{31}{4}}$.
\begin{proof}[Proof of Claim~\ref{I=S=1}]

If $\sigma\ge \omega-2$, by $\sigma\tau+\omega\tau+\omega\ge 2k-1.$, we will have $(\sigma+\omega)(\tau+\frac{1}{2})\ge 2k-1-\frac{\omega-\sigma}{2}=2k-2$, hence $(\sigma+\omega+\tau+\frac{1}{2})^2\ge 8k-4+(\sigma+\omega-\tau-\frac{1}{2})^2\ge 8k-\frac{31}{4}$.
Then we are done.

If $\sigma\le \omega-3$, since $\sigma\tau+\tau+\omega\ge k$, we have  
$$2k\le (\omega+\sigma-1)\tau+(\omega+\sigma-1)+4-(\omega-\sigma-3)(\tau-1)\le (\omega+\sigma-1)(\tau+1)+4.$$
This will implies $\omega+\sigma+\tau\ge \sqrt{8k-16}\ge \sqrt{8k-\frac{31}{4}}-\frac{1}{2}$. The last inequality is ture when $k\ge 10$.

So we always have $\omega+\sigma+\tau\ge \sqrt{8k-\frac{31}{4}}-\frac{1}{2}$, and $n=5k-3-(\omega+\sigma+\tau)\le 5k-2.5-\sqrt{8k-\frac{31}{4}}$.

\end{proof}

\begin{center}
\resizebox{13.5cm}{!}
{
\begin{tikzpicture}

%\draw [pattern={dots},pattern color=gray] (-5,2.5) ellipse (1.5 and 1);
\filldraw [fill = gray!50] (-5,2.5) ellipse (1.5 and 1);
%\draw (-5,2.5) ellipse (1.5 and 1);
\node at (-5,2.6) {$V-A-U$};
\node at (-5,2.2) {$k-1-\tau$};

\fill [fill = gray!50] (-7.6,0.35) ellipse (1.5 and 1);
\fill [fill = white] (-7.3,1.35) rectangle (-6.1,-0.65); 
\draw (-7.6,0.35) ellipse (1.5 and 1);
\draw (-7.6,0.35) ellipse (1.5 and 1);
\node at (-6.2,-0.45) {$U_{t}$};
\draw  plot[smooth] coordinates {(-7.3,1.35) (-7.3,-0.65)};
\node at (-6.7,0.35) {$U_t \cap A$};
\node at (-6.75,-0.05) {$\sigma$};
\node at (-8.1,0.05) {$k-1-\sigma$};

\filldraw [fill = gray!50] (-2.45,0.35) ellipse (1.5 and 1);
\node at (-1,-0.45) {$U'$};
\draw  plot[smooth] coordinates {(-2.1,1.35) (-2.1,-0.65)};
\node at (-2.9,0.35) {$U'_{a}$};
\node at (-2.9,-0.05) {$\sigma\tau$};
\node at (-1.65,0.6) {$U'_{b}$};
\node at (-1.55,0.25) {$k-\omega$};
\node at (-1.65,-0.1) {$-\sigma\tau$};

\draw  (-6.7,-2.5) ellipse (1.5 and 1);
\node at (-5.5,-1.55) {$A'$};
\node at (-6.7,-2.8) {$k-1$};

\filldraw [fill = gray!50] (-3.3,-2.5) ellipse (1.5 and 1);
\node at (-2.1,-1.55) {$A_{k}$};
\draw  plot[smooth] coordinates {(-3,-1.5) (-3,-3.5)};
\node at (-3.8,-2.8) {$k-\sigma-\omega$};
\node at (-2.4,-2.5) {$A_k \cap U$};
\node at (-2.4,-2.9) {$\omega$};
\fill (-2.3, -2) circle(0.05);
\node at (-2.6,-2) {$z$};

\draw  plot[smooth] coordinates {(-6.5,2.5) (-7.3,1.35)}; 
\draw  plot[smooth] coordinates {(-3.5,2.5) (-2.1,1.35)}; 
\draw  plot[smooth, tension = 1.5] coordinates {(-8.6,1.1) (-5.2,4) (-2.1,1.35)}; 
\draw[dashed]  plot[smooth] coordinates {(-6.15,0.55) (-3.9,0.55)}; 
\node at (-5.9,0.75) {$\tau$};
\node at (-4.1,0.75) {$1$};
\draw  plot[smooth] coordinates {(-7.3,-0.65) (-3,-1.5)};
\draw  plot[smooth] coordinates {(-2.1,-0.65) (-7,-1.5)};
\draw[dashed]  plot[smooth, tension = 1.7] coordinates {(-7.7,-3.25) (-5,-4) (-2.3,-3.25)};
\node at (-7.4,-4) {$\omega-1$};
\node at (-2.5,-4) {$k-1-\tau$};
\draw[dashed]  plot[smooth, tension = 1] coordinates {(-2.3,-2) (-0.8,-3.05) (-2,-4.5) (-5,-4)};
\node at (-2,-4.8) {$(\omega-1)\tau$};

\node at (-5,-4.7) {\LARGE $R$};

\draw [fill = gray!50] (5,2.5) ellipse (1.5 and 1);
%\draw  (5,2.5) ellipse (1.5 and 1);
\node at (5,2.6) {$V-A-U$};
\node at (5,2.2) {$k-1-\tau$};

\filldraw [fill = gray!50]  (2.4,0.35) ellipse (1.5 and 1);
\node at (3.8,-0.45) {$U_{t}$};
\draw  plot[smooth] coordinates {(2.7,1.35) (2.7,-0.65)};
\node at (3.3,0.35) {$U_t \cap A$};
\node at (3.25,-0.05) {$\sigma$};
\node at (1.9,0.05) {$k-1-\sigma$};

\draw  (7.55,0.35) ellipse (1.5 and 1);
\node at (9,-0.45) {$U'$};
\draw  plot[smooth] coordinates {(7.9,1.35) (7.9,-0.65)};
\node at (7.1,0.35) {$U'_{a}$};
\node at (7.1,-0.05) {$\sigma\tau$};
\node at (8.35,0.6) {$U'_{b}$};
\node at (8.45,0.25) {$k-\omega$};
\node at (8.35,-0.1) {$-\sigma\tau$};

\filldraw [fill = gray!50] (3.3,-2.5) ellipse (1.5 and 1);
\node at (4.5,-1.55) {$A'$};
\node at (3.3,-2.8) {$k-1$};

\fill [fill = gray!50] (6.7,-2.5) ellipse (1.5 and 1);
\fill [fill = white] (7,-1.5) rectangle (8.2,-3.5); 
\draw  (6.7,-2.5) ellipse (1.5 and 1);
\node at (7.9,-1.55) {$A_{k}$};
\draw  plot[smooth] coordinates {(7,-1.5) (7,-3.5)};
\node at (6.2,-2.8) {$k-\sigma-\omega$};
\node at (7.6,-2.5) {$A_k \cap U$};
\node at (7.6,-2.9) {$\omega$};
\fill (7.7, -2) circle(0.05);
\node at (7.4,-2) {$z$};

\draw  plot[smooth, tension = 1] coordinates {(4.8,1.5) (4.2,-0.3) (3.7,-1.55)}; 
\draw  plot[smooth, tension = 1] coordinates {(5.2,1.5) (6,-0.3) (7,-1.5)}; 
\draw  plot[smooth] coordinates {(4.5,-1.9) (5.4,-1.9)}; 
\draw  plot[smooth, tension = 1.5] coordinates {(2.9,1.3) (5.2,4) (8.55,1.1)}; 
\draw[dashed]  plot[smooth, tension = 1.5] coordinates {(3,1.3) (5,3.8) (7.55,1.35)}; 
\node at (3.65,1.5) {$\sigma\tau-\tau$};
\node at (7.3,1.6) {$\sigma-1$};
\draw  plot[smooth] coordinates {(2.1,-0.65) (3,-1.5)};
\draw  plot[smooth] coordinates {(7.9,-0.65) (7,-1.5)};
\draw[dashed]  plot[smooth, tension = 1.7] coordinates {(2.3,-3.25) (5,-4) (7.7,-3.25)};
\node at (2.6,-4) {$1$};
\node at (7.1,-4) {$\tau$};
\draw[dashed]  plot[smooth, tension = 1] coordinates {(7.7,-2) (9.2,-3.05) (8,-4.5) (5,-4)};
\node at (8,-4.8) {$k-1-(\omega-1)\tau$};

\node at (5,-4.7) {\LARGE $B$};

\draw (-4.5,-5.7) rectangle (-3.5,-6.3);
\node at (-2.3, -6) {independent};
\filldraw [fill = gray!50]  (-0.5,-5.7) rectangle (0.5,-6.3);
\node at (1.5, -6) {complete};
\node at (4.5, -6) {$z = a_k^{\omega} = u_{\sigma\tau+1}$};

\end{tikzpicture}
}
\end{center}

\section{The counterexample}

In this section, we illustrate a sharp example for Theorem \ref{main} (2). 
That is, for given integer $k$ and $n$, 
where $4k-3 \le n = \lfloor 5k-2.5-\sqrt{8k-\frac{31}{4}} \rfloor$,
we construct a 2-edge-colored $K_n$, 
which contains no $k$-connected monochromatic subgraph with at least $n-2k+2$ vertices. 
Note that we may assume $k \ge 6$, 
otherwise, the example with $n = 4k-4$ vertices, 
which was mentioned in \cite{BG08} and \cite{Mat83}, 
can serve as the counterexample we desire. 
Let $G = (V, E)$ be a complete graph on $n$ vertices. 
We demonstrate the coloring of our example 
by specifying and verifying the two strong $(2k-2,k)$-decompositions: 
$((A_i, C_i, D_i))_{i=1}^{l_R}$ for the red graph $G_R$, 
and $((U_s, X_s, Y_s))_{s=1}^{l_B}$ for the blue graph $G_B$. 
We then obtain $G_R$ and $G_B$ 
by coloring all AA-type and AC-type edges red, 
and all UU-type and UX-type edges blue. 
and justify that every edge in $G$ receives at least one color. 
Note that some of the edges might be colored red and blue simultaneously. 
The example was inspired by the case $|I^*| = |S^*| = 1$ in our proof. 

Suppose $5k-3-n=\lceil \sqrt{8k-\frac{31}{4}}-\frac{1}{2}\rceil=4x+b$, where $x, b\in \mathbb{Z}$ with $b\in \{-2,-1,0, 1\}$. That is, $x=\lceil\frac{ \sqrt{8k-\frac{31}{4}}-\frac{3}{2}}{4}\rceil$, and $b=5k-3-n-4x$. 
Note that $x \ge 2$, since $k \ge 6$. 
Let $\sigma=x-1, \omega=x+1, \tau=2x+b$. Then $5k-3-n=\sigma+\omega+\tau$.

We have $$k\le \frac{(4x+b+\frac{1}{2})^2+\frac{31}{4}}{8}=x(2x+b+\frac{1}{2})+\frac{(b+\frac{1}{2})^2+\frac{31}{4}}{8}\le x(2x+b+\frac{1}{2})+\frac{5}{4}.$$ 

Therefore $k\le \lfloor x(2x+b)+\frac{x}{2}+\frac{5}{4}\rfloor\le x(2x+b)+\frac{x}{2}+1=\frac{\sigma+\omega}{2}\tau+\frac{\omega+1}{2}.$ Furthermore $x(2x+b)+\frac{x}{2}+1\le x(2x+b)+x+1=(\sigma+1)\tau+\omega$. Hence 
we will have $\sigma\tau+\tau+\omega\ge k$ and $\sigma\tau+\omega\tau+\omega\ge 2k-1$.
Moreover, this implies $\omega\tau \ge k-1$ 
since $\omega\tau + 1 = \sigma\tau + 2\tau + 1 \ge \sigma\tau + \tau + \omega \ge k$. 

On the other hand, $\sqrt{8k-\frac{31}{4}}+\frac{1}{2}> 4x+b$. Therefore, when $x \ge 2$, 
$k>\frac{(4x+b-\frac{1}{2})^2+\frac{31}{4}}{8}=x(2x+b-\frac{1}{2})+\frac{(b-\frac{1}{2})^2+\frac{31}{4}}{8}=((x-1)(2x+b)+x)+\frac{x}{2}+b+\frac{b^2-b+8}{8}
= \sigma\tau+\omega-1+\frac{x}{2}+\frac{b^2+7b+8}{8}
= \sigma\tau+\omega+\frac{x}{2}+\frac{b^2+7b}{8}$. 
If  $(x, b) \ne (2, -2)$, 
then $k> \sigma\tau+\omega+\frac{x}{2}+\frac{b^2+7b}{8}
\ge \sigma\tau+\omega$. 
Futhermore, when $(x, b) = (2, -2)$, 
$k\ge6>5=\sigma\tau+\omega$. 
Hence we have $k \ge \sigma\tau+\omega+1$.

Let $t = k-\omega+2$. 
We first define four disjoint vertex sets: $A'$, $A_k$, $U'$, and $U_t$. 
Let $A' = \{a_1, \dots, a_{k-1}\}$, 
$A_k = \{a_k^1, \dots, a_k^{k-\sigma}\}$, 
$U' = U'_{a} \cup U'_{b}$ 
where $U'_{a} = \{u_1, \dots, u_{\sigma\tau}\}$ 
and $U'_{b} = \{u_{\sigma\tau+2}, \dots, u_{k-\omega+1}\}$, 
and $U_t = \{u_t^1, \dots, u_t^{k-1}\}$. 
Note that $U'_{b}$ is well-defined and non-empty 
since $k \ge \sigma\tau + \omega + 1$. 
Moreover, we let $A = A' \cup A_k \cup \{u_t^1, \dots, u_t^\sigma\}$, 
and $U = U' \cup U_t \cup \{a_k^1, \dots, a_k^\omega\}$. 
Here $A$ and $U$ are both well-defined, 
since $\sigma \le k-1$, and $\omega \le k-\sigma\tau-1 \le k-\sigma$. 
Note that $A \cap U$ consists of two parts: 
$A_k \cap U = \{a_k^1, \dots, a_k^\omega\}$ and $U_t \cap A = \{u_t^1, \dots, u_t^\sigma\}$. 
For convenience, let $z = a_k^{\omega} = u_{\sigma\tau+1}$. 
Besides $A'$, $A_k$, $U'$, and $U_t$, we still have 
$5k-3-(\sigma+\omega+\tau)-(k-1)-(k-\sigma)-(k-\omega)-(k-1) = k-1-\tau$
vertices left. 
Let $V-A-U = \{v_1, \dots, v_{k-1-\tau}\}$. 

We now set $A_i, C_i, D_i$ for $i \in [k+\sigma]$ to construct $G_R$. 
For $i \in [k-1]$, 
we set $A_i = \{a_i\}$, 
and $C_i = U' \cup (A_k \cap U \setminus \{a_k^{\lceil\frac{i}{\tau}\rceil}\})$. 
Note that $\omega-1 = \sigma+1 \le \lceil\frac{k-1}{\tau}\rceil \le \omega$, 
since $\sigma\tau+\omega \le k-1 \le \omega\tau$. 
We set $C_k = U_t$. 
For $i=k+i'$ where $i' \in [\sigma]$, 
We set $A_i = \{u_t^{i'}\}$, 
and let $C_i$ contains all vertices in $V-A-U$, 
and $\tau$ vertices in $U'_a$, 
such that $C_i \cap U'_a=\{u_{(i'-1)\tau+1}, \dots, u_{i'\tau}\}$. 
We set $D_i = V(G) \setminus (\bigcup_{i'=1}^i A_{i'} \cup C_i)$ for $i \in [k+\sigma]$. 
The red graph $G_R$ consists of all AA-type and AC-type edges. 

Next, we set $U_s, X_s, Y_s$ for $s \in [k+1]$ to construct $G_B$. 
For $s \in [t-1]$, 
we set $U_s = \{u_s\}$. 
For $s \in [\sigma\tau]$, 
where $u_s \in U'_{a}$, 
let $X_s = A_k \cup (U_t \cap A \setminus \{u_t^{\lceil \frac{s}{\tau}\rceil}\})$. 
Notice that for each $s \in [\sigma\tau]$, 
$u_su_t^{\lceil \frac{s}{\tau}\rceil}$ is always an AC-type edge. 
For $s = \sigma\tau+1$, 
where $u_s = z$, 
let $X_{\sigma\tau+1} = (V-A-U) \cup U'_{b} \cup \{a_{\tau(\omega-1)+1}, \dots, a_{k-1}\}$. 
Note that $|X_{\sigma\tau+1}| = (k-1-\tau)+(k-\omega-\sigma\tau)+(k-1-\tau(\omega-1)) 
= 3k - 2 - \sigma\tau - \omega\tau - \omega
\le k-1$, 
as $\sigma\tau+\omega\tau+\omega\ge 2k-1$. 
For $s \in [\sigma\tau+2, t-1]$, 
where $u_s \in U'_{b}$, 
let $X_s = A_k \cup (U_t \cap A \setminus \{z\})$. 
We set $X_t = A'$. 
For $s = t+s'$ where $s' \in [\omega-1]$, 
we set $U_s = \{a_k^{s'}\}$, 
and let $X_s$ contains all vertices in $V-A-U$, 
and $\tau$ vertices in $A'$, 
such that $X_s \cap A'=\{a_{(s'-1)\tau+1}, \dots, a_{s'\tau}\}$. 
Notice that this covers all the non-neighbours of $a_k^{s'}$ in $A'$ in $G_R$. 
We set $Y_s = V(G) \setminus (\bigcup_{s'=1}^s U_{s'} \cup X_s)$ for $s \in [k+1]$. 
The blue graph $G_B$ consists of all UU-type and UX-type edges. 

By definition, 
it is not difficult to see that $((A_i, C_i, D_i))_{i=1}^{k+\sigma}$ is a strong $(2k-2, k)$-decomposition of $G_R$, 
and $((U_s, X_s, Y_s))_{s=1}^{k+1}$ is a strong $(2k-2, k)$-decomposition of $G_B$. 
Moreover, we can verify $R \cup B$ covers all edges. 
Hence, there does not exist a $k$-connected monochromatic subgraphs with at least $n-2k+2$ vertices,  
Thus, the example we propose confirmed Theorem \ref{main}(2).

\section{Conclusion}

In this paper, 
we presented a counterexample of Bollob\'{a}s and Gy\'{a}rf\'{a}s' conjecture with 
$n = \lfloor 5k-2.5-\sqrt{8k-\frac{31}{4}} \rfloor$. 
We also verified the conjecture for $n > 5k-2.5-\sqrt{8k-\frac{31}{4}}$ for $k \ge 16$. 
We believe the requirement of $k\ge 16$ can be relaxed. 
We also provided a shorter proof for $n \ge 5k-\min\{\sqrt{4k-2}+3,0.5k+4\}$ and a simpler counter-example with order $n=5k-3-2\lceil\sqrt{2k-1}\rceil$ in an earlier version of this paper, which could be found on Arxiv \cite{2008.09001}.

%We have found some examples, 
%which are very different from the counterexample we raised in section 3, 
%and are not covered by the inequality in section 4. 
%However, none of them could serve as a counterexample to the statement. 
%Thus, we conjecture that: 
%
%\begin{conj}
%Let $k,n \in \mathbb{Z}^+$. 
%If $n > 5k-\lfloor2\sqrt{2k-1}\rfloor-3$, 
%then for any 2-edge-colored $K_n$, 
%there exists a $k$-connected monochromatic subgraph, which contains at least $n-2k+2$ vertices. 
%\end{conj}

Recall that Matula showed that when $n> (3+\sqrt{11/3})(k-1)\approx 4.91k$, any 2-edge-coloring of $K_n$ has a $k$-connected monochromatic subgraph. Even though our requirement for $n$ is a little more strict, our result can satisfy the restriction on the order of the $k$-connected monochromatic subgraph. Probably our technique is enough to improve the result of Matula.

\begin{comment}
More generally, consider the statement 
`` For $k, n \in \mathbb{Z}^+$ with $n \ge g(k)$, 
every 2-edge-colored $K_n$ must contain a $k$-connected monochromatic subgraph 
with at least $n-f(k)$ vertices". 
For a given $f(k) \ge 2k-2$, what is the minimum $g(k)$ for the statement to be true? 
Note that when $f(k) \le 2k-1$, the example $B(n, k)$ in \cite{BG08} can always serve as a counterexample of the statement. 
On the other hand, if $g(k) \in [4k-3, 5k-4]$ is fixed, what is the correlated $f(k)$? 
In other words, given the number of vertices $n$, 
what is the order of the largest $k$-connected monochromatic subgraph we can guarantee in a 2-edge-colored $K_n$?
\end{comment}

Moreover, our result improves the bounds for some other related problems. 
For example, since every $k$-connected graph has minimum degree at least $k$, 
Theorem \ref{main} (1) leads to the following corollary: 

\begin{coro}
If $n \ge 5k-2.5-\sqrt{8k-\frac{31}{4}}$ where $k$ is sufficiently large, 
then for any 2-edge-colored $K_n$, 
there exists a monochromatic subgraph with minimum degree at least $k$, 
which contains at least $n-2k+2$ vertices. 
\end{coro}

This problem concerning monochromatic large subgraphs with a specified minimum degree in edge-colored graphs 
has been studied by Caro and Yuster \cite{CY03}. 
By applying their conclusion on 2-edge-colored complete graphs, 
the corollary holds when $n \ge 7k+4$, 
which could be covered by our result.

Furthermore, 
there are some open problems related to Bollob\'{a}s and Gy\'{a}rf\'{a}s' conjecture, 
such as the multicoloring version of the conjecture, 
and forcing large highly connected subgraphs with given independence number. 
We believe the decomposition and calculation technique we introduced in this paper could also be applied to improve the results of those topics.

\section{Acknowledgments}

This work was partially supported by the National Natural Science Foundation of China [Grant No.11931006, 12201390], the National Key R\&D Program of China [Grant No. 2020YFA0713200, 2022YFA1006400], and the Shanghai Dawn Scholar Program [Grant No. 19SG01]. 

We would like to show our gratitude to Prof. Xingxing Yu, Georgia Institute of Technology, 
for sharing his pearls of wisdom with us during the course of this research. 
We are immensely grateful to Henry Liu, Sun Yat-sen University, 
for his comments on an earlier version of the manuscript. 
We would also like to express our thanks to Chaoliang Tang, Fudan University for his careful proofreading.

%\bibliographystyle{plain}%
%\bibliographystyle{elsarticle-num}%
%\bibliography{bibfile}

\end{document}